\documentclass{amsart}
\usepackage{amsmath}
\usepackage{amsthm}
\usepackage{listings}
\usepackage{xcolor,soul}
\usepackage{tikz-cd}
\usetikzlibrary{decorations.markings}
\tikzset{negated/.style={
        decoration={markings,
            mark= at position 0.5 with {
                \node[transform shape] (tempnode) {$\backslash$};
            }
        },
        postaction={decorate}
    }
}
\usepackage{mathtools}
\usepackage{extpfeil}

\newtheorem{theorem}{Theorem}[section]
\newtheorem{proposition}[theorem]{Proposition}
\newtheorem{lemma}[theorem]{Lemma}
\usepackage{amssymb}

\usepackage{algorithm2e}
\usepackage{enumitem}  
\usepackage{stmaryrd}

\setenumerate{label={\normalfont(\textit{\roman*})},ref={\textrm{\roman*}},}
\setenumerate{label={\normalfont(\textit{\roman*})}, ref={\textrm{\roman*}},
wide, leftmargin=*, labelwidth=!, labelindent=0pt, 
}

\usepackage{ifthen}
\usepackage{scalerel}
\usepackage{hyperref}
\usepackage{mathrsfs}
\usepackage[mathscr]{euscript}
\usepackage{lineno}

\newtheorem{corollary}[theorem]{Corollary}

\newtheorem{alphatheorem}{Theorem}

\theoremstyle{definition}
\newtheorem{definition}[theorem]{Definition}
\newtheorem{example}[theorem]{Example}

\newtheorem*{remark*}{Remark}

\newtheorem{question}[theorem]{Question}

\DeclareMathOperator*{\PPP}{\scalerel*{\mathbb{P}}{\textstyle\sum}}

\newcommand{\ip}[1]{\left\lfloor #1 \right\rfloor }

\newcommand{\floor}[1]{\left\lfloor #1 \right\rfloor}
\newcommand{\bra}[1]{\left(#1\right)}
\newcommand{\brabig}[1]{\big(#1\big)}

\newcommand{\braBig}[1]{\Big( #1 \Big)}
\renewcommand{\tilde}{\widetilde}
\renewcommand{\bar}{\overline}

\newcommand{\abs}[1]{\left|#1\right|}

\newcommand{\set}[2]{\left\{ #1 \ \middle| \ #2 \right\} }

\newcommand{\ceil}[1]{\left\lceil #1 \right\rceil}

\newcommand{\e}{\varepsilon}
\renewcommand{\a}{\alpha}
\renewcommand{\b}{\beta}

\newcommand{\NN}{\mathbb{N}}

\newcommand{\QQ}{\mathbb{Q}}

\newcommand{\ZZ}{\mathbb{Z}}
\newcommand{\RR}{\mathbb{R}}

\newcommand{\cA}{\mathcal{A}}

\newcommand{\cH}{\mathcal{H}}

\newcommand{\cN}{\mathcal{N}}

\definecolor{fresh}{HTML}{2bb101}
\definecolor{checked}{HTML}{1e5e06}
\definecolor{double}{HTML}{5E3800}
\definecolor{external}{HTML}{a81a78}
\definecolor{later}{HTML}{0410ff}
\definecolor{minor-rev}{HTML}{d96a09}
\definecolor{major-rev}{HTML}{c90000}
\definecolor{skip}{HTML}{ffffff}\definecolor{normal}{HTML}{000000}

\definecolor{fresh}{HTML}{000000}
\definecolor{checked}{HTML}{000000}
\definecolor{double}{HTML}{000000}
\definecolor{external}{HTML}{000000}
\definecolor{later}{HTML}{000000}
\definecolor{minor-rev}{HTML}{000000}
\definecolor{major-rev}{HTML}{000000}

\newcommand{\bb}{\mathbf}

\renewcommand{\subset}{\subseteq}

\newcommand*\patchAmsMathEnvironmentForLineno[1]{\expandafter\let\csname old#1\expandafter\endcsname\csname #1\endcsname
  \expandafter\let\csname oldend#1\expandafter\endcsname\csname end#1\endcsname
  \renewenvironment{#1}{\linenomath\csname old#1\endcsname}{\csname oldend#1\endcsname\endlinenomath}}\newcommand*\patchBothAmsMathEnvironmentsForLineno[1]{\patchAmsMathEnvironmentForLineno{#1}\patchAmsMathEnvironmentForLineno{#1*}}\AtBeginDocument{\patchBothAmsMathEnvironmentsForLineno{equation}\patchBothAmsMathEnvironmentsForLineno{align}\patchBothAmsMathEnvironmentsForLineno{flalign}\patchBothAmsMathEnvironmentsForLineno{alignat}\patchBothAmsMathEnvironmentsForLineno{gather}\patchBothAmsMathEnvironmentsForLineno{multline}}

\newcounter{claimcounter}
\counterwithin*{claimcounter}{theorem}

\begin{document}

\author[J.\ Konieczny ]{Jakub Konieczny}

\title{An asymptotic version of  Cobham's theorem}
\begin{abstract}
	We introduce the notion of an asymptotically automatic sequence, which generalises the notion of an automatic sequence, and we prove a variant of Cobham's theorem for the newly introduced class of sequences.
\end{abstract}


\maketitle 

\color{checked}
\section{Introduction}\label{sec:Intro}

\subsection{Background}
\color{checked}
Cobham's theorem \cite{Cobham-1969} is one of the fundamental results in the theory of automatic sequences, that is, sequences whose $n$-th term can be computed by a finite automaton which receives as input the expansion of $n$ in a given base $k \geq 2$. Said theorem asserts that a sequence cannot be automatic with respect to two different bases $k$ and $l$, except for the arguably trivial cases where the sequence is eventually periodic (hence automatic in all bases), or where the bases are multiplicatively dependent (hence $k$-automatic sequences are the same as $l$-automatic sequences). We recall that two integers $k,l \geq 2$ are multiplicatively dependent if they are both powers of the same integer, that is, if $(\log k)/(\log l) \in \QQ$, and they are multiplicatively independent otherwise. For background on automatic sequences, we refer to \cite{AlloucheShallit-book}.

Cobham's theorem has attracted considerable attention over the years, and has been extended in many directions. Specifically, variants of this theorem are known for
the multidimensional setting \cite{Semenov-1977} (see also \cite{Durand-2008} and \cite{MichauxVillemaire-1996});
other numeration systems \cite{ Fabre-1994, PointBruyere-1997, Durand-1998b, Hansel-1998,  Bes-2000, DurandRigo-2009};
morphic sequences \cite{Durand-2011} (see also \cite{ Fabre-1994, Durand-1998, Durand-2002-AA});
fractals \cite{AdamczewskiBell-2011} (see also \cite{ChanHare-2014} and \cite{CharlierLeroyRigo-2015});
quasi-automatic sequences \cite{AdamczewskiBell-2008};
regular sequences \cite{Bell-2005};
 Mahler series \cite{AdamczewskiBell-2013} (see also \cite{SchafkeSinger-2017});
real numbers \cite{BoigelotBrusten-2009} (see also \cite{ BoigelotBrustenLeroux-2009, BoigelotBrustenBruyere-2010});
Gaussian integers \cite{HanselSafer-2003,BosmaFokkinkKrebs-2017}. In joint work with J. Byszewski \cite{ByszewskiKonieczny-2020-AA}, motivated by a question of Deshouillers \cite{DeshouillersRuzsa-2011,Deshouillers-2012}, we obtained a variant of Cobham's theorem up to equality almost everywhere. 

\subsection{Asymptotic automaticity}
The purpose of this paper is to add another entry to the list of extensions of Cobham's theorem, and to obtain a strictly stronger variant of the main result of \cite{ByszewskiKonieczny-2020-AA}. In order to state it, we will need to introduce some terminology. 

We will say that two sequences $f,g \colon \NN_0 \to \Omega$ taking values in a set $\Omega$ are \emph{asymptotically equal}, denoted by $f \simeq g$, if $f(n) = g(n)$ for almost all $n \in \NN_0$, that is,
\[
	\abs{ \set{n < N}{f(n) \neq g(n)}}/N \to 0 \text{ as } N \to \infty.
\]

For $k \in \NN$ with $k \geq 2$ and for a sequence $f \colon \NN_0 \to \Omega$ we define the \emph{$k$-kernel} of $f$, which consists of all restrictions of $f$ to a residue class modulo a power of $k$,
\[
	\cN_k(f) = \set{ \NN_0 \ni n \mapsto f(k^\a n + r) \in \Omega}{ \a, r \in \NN_0,\ r < k^\a}.
\]	
The sequence $f$ is $k$-automatic if and only if its $k$-kernel is finite, $\abs{\cN_k(f)} < \infty$. 

\begin{definition}
	Let $k \in \NN$, $k \geq 2$, and let $f \colon \NN_0 \to \Omega$ be a sequence taking values in a finite set $\Omega$. Then $f$ is \emph{asymptotically $k$-automatic} if and only if $\abs{ \cN_k(f)/{\simeq} } < \infty$.
\end{definition}

 In other words, the sequence $f$ is asymptotically $k$-automatic if and only if there exists a finite number of sequences $f_0,f_1,\dots,f_{d-1} \colon \NN_0 \to \Omega$ such that for each $f' \in \cN_k(f)$ there exists $0 \leq i < d$ such that $f' \simeq f_i$. An alternative description of asymptotically $k$-automatic sequences is given in Lemma \ref{lem:ass-aut=>rep}. The requirement that the output space $\Omega$ should be finite does not play a significant role, but we include it for the sake of consistency with the theory of automatic sequences.

Of course, each $k$-automatic sequence is asymptotically $k$-automatic. In \cite{ByszewskiKonieczny-2020-AA}, we studied sequences which are asymptotically equal to an automatic sequence, that is, sequences $f \colon \NN_0 \to \Omega$ such that there exists a $k$-automatic sequence $\tilde f \colon \NN_0 \to \Omega$ with $f \simeq \tilde  f$. Other examples of asymptotically $k$-automatic sequences appear in an upcoming paper on classification of automatic subsemigroups of $\NN$ \cite{KlurmanKonieczny-upcoming}. In the situation above, $f$ is asymptotically $k$-automatic; in fact, it is not hard to check that $\abs{\cN_k(f)/{\simeq}} \leq \abs{\cN_k(\tilde f)}$. We note, however, that the reverse implication does not hold, as shown by the following examples. To avoid breaking the flow of the exposition, we will postpone the proofs of claims made in Examples \ref{ex:leading-prime} and \ref{ex:gap} to Section \ref{sec:Example-2}. 

\newcommand{\lini}{\lambda}\newcommand{\lmax}{\kappa}

\color{fresh}
\begin{example}\label{ex:leading-prime}\color{checked}

	For an integer $n \in \NN_0$, let $\lini(n)$ denote the number of leading $1$s in the binary expansion of $n$. For instance, $\lini(0) = 0$ and $\lini(123) = \lini\bra{[1111011]_2} = 4$. Define $f \colon \NN_0 \to \{0,1\}$ by
	\[
		f(n) =
	\begin{cases}
		1 & \text{if } \lini(n) \text{ is prime},\\
		0 & \text{otherwise.}
	\end{cases}
	\]
	Then $f$ is asymptotically $2$-automatic but not asymptotically equal to a $2$-automatic sequence.
\end{example}	

{\color{minor-rev}

}

\begin{example}\label{ex:gap}
	For an integer $n \in \NN_0$, let $\lmax(n)$ denote the maximal number of consecutive $1$s in the binary expansion of $n$. For instance, $\lmax(0) = 0$ and $\lmax(1234) = m\bra{[10011010010]_2} = 2$. 
Define $f \colon \NN_0 \to \{0,1\}$ by
	\[
		f(n) =
	\begin{cases}
		1 & \text{if } \lmax(n) \text{ is odd},\\
		0 & \text{otherwise.}
	\end{cases}
	\]	
	Then $f$ is asymptotically $2$-automatic but not asymptotically equal to a $2$-automatic sequence. 
\end{example}

\color{checked}

We will also need two asymptotic analogues of the notion of a periodic sequence.
For a sequence $f \colon \NN_0 \to \Omega$, we let $S f$ denote the shift of $f$ by $1$, that is, $S f(n) = f(n+1)$ for all $n \in \NN_0$. It is almost a tautology to say that $f$ is periodic with period $m$ (or $m$-periodic, for short) if and only if $S^mf = f$. We will say that $f \colon \NN_0 \to \Omega$ is \emph{asymptotically equal to an $m$-periodic sequence} if there exists an $m$-periodic sequence $\tilde f$ with $f \simeq \tilde f$, and we will say that $f$ is \emph{asymptotically invariant under shift by $m$} if $S^m f \simeq f$. It is not hard to check that each asymptotically $m$-periodic sequence is asymptotically invariant under shift by $m$, but the reverse implication does not hold, and asymptotic invariance under a shift does not imply asymptotic automaticity. Like before, we defer the proof of the following example to Section \ref{sec:NonExple}.

\begin{example}\label{ex:non}
Define $f \colon \NN_0 \to \{0,1\}$ by
	\[
		f(n) = \ip{\sqrt{n}} \bmod 2.
	\]	
	Then $f$ is asymptotically invariant under shift by $1$ but not asymptotically $k$-automatic in any base $k \geq 2$, and hence not asymptotically periodic.
\end{example}

\subsection{New results}
We are now ready to formulate our main results.

\begin{alphatheorem}\label{thm:main-symmetric}
	Let $\Omega$ be a finite set, let $k,l \geq 2$ be multiplicatively independent integers and let $f \colon \NN_0 \to \Omega$ be a sequence. Suppose that $f$ is asymptotically $k$-automatic and asymptotically $l$-automatic. Then $f$ is asymptotically invariant under a non-zero shift.
\end{alphatheorem}

\begin{alphatheorem}\label{thm:main-asymmetric}
	Let $\Omega$ be a finite set, let $k,l \geq 2$ be multiplicatively independent integers and let $f \colon \NN_0 \to \Omega$ be a sequence. Suppose that $f$ is $k$-automatic and asymptotically $l$-automatic. Then $f$ is asymptotically equal to a periodic sequence.
\end{alphatheorem}

We note that the main result of \cite{ByszewskiKonieczny-2020-AA} follows directly from Theorem \ref{thm:main-asymmetric}.
\begin{corollary}[\cite{ByszewskiKonieczny-2020-AA}]
	Let $\Omega$ be a finite set, let $k,l \geq 2$ be multiplicatively independent integers and let $f,g \colon \NN_0 \to \Omega$ be sequences. Suppose that $f$ is $k$-automatic, $g$ is $l$-automatic and $f \simeq g$. Then $f$ and $g$ are asymptotically equal to a periodic sequence.
\end{corollary}
\begin{proof}
	Since $f \simeq g$ and $g$ is $l$-automatic, $f$ is asymptotically $l$-automatic. Hence, $f$ is asymptotically equal to a periodic sequence by Theorem \ref{thm:main-asymmetric}.
\end{proof}

\color{fresh}

It is natural to ask if one can obtain a joint generalisation of Theorems \ref{thm:main-symmetric} and \ref{thm:main-asymmetric}, asserting that a sequence that is asymptotically automatic in two multiplicatively independent bases must be asymptotically equal to a periodic sequence. It turns out that such a generalisation is not possible, and in Section \ref{sec:Example} we give an example of a sequence that is asymptotically automatic in bases $2$ and $3$, but not asymptotically equal to a periodic sequence. 

\subsection{Future directions}

We recall that, as a direct consequence of Cobham's theorem, a sequence that is automatic in two multiplicatively independent bases is automatic in \emph{every} base. Similarly, it follows from Theorem \ref{thm:main-asymmetric} that a sequence that is automatic and asymptotically automatic in two multiplicatively independent bases is also asymptotically automatic in every base. On the other hand, we believe that the aforementioned sequence constructed in Section \ref{sec:Example} is not asymptotically $k$-automatic for any $k \in \NN$ with $\gcd(k,6) = 1$ (the proof of this statement remains elusive). Similarly, it is straightforward to generalise the construction in Section \ref{sec:Example} to more than two bases. However, a proof that the resulting sequence is asymptotically automatic in the appropriate bases would require either much more precise Diophantine approximation estimates or a completely new approach. This motivates us to pose the following questions, which are the simplest instances that are currently open.

\begin{question}
	Does there exist a sequence which is asymptotically automatic in bases $2$ and $3$, but not in base $5$?
\end{question}

\begin{question}
	Does there exist a sequence which is asymptotically automatic in bases $2$, $3$ and $5$, but not asymptotically equal to a periodic sequence?
\end{question}

\subsection*{Acknowledgements}
The author wishes to thank Boris Adamczewski, Oleksiy Klurman, Piotr Migda\l, Clemens M\"ullner and Lukas Spiegelhofer for helpful comments, and Terence Tao for providing the answer to the MathOverflow question \cite{Tao-MO}. 
The author works within the framework of the LABEX MILYON (ANR-10-LABX-0070) of Universit\'e de Lyon, within the program "Investissements d'Avenir" (ANR-11-IDEX- 0007) operated by the French National Research Agency (ANR). 
\section{Preliminaries}

\subsection{Density}\color{checked}
We let $\NN = \{1,2,3,\dots\}$ denote the set of positive integers and $\NN_0 = \NN \cup \{0\}$. By a slight abuse of notation, for $a,b \in \ZZ$ we let $[a,b) = \{a,a+1,\dots,b-1\}$ denote the integer (rather than real) interval. For a set $A \subset \NN_0$, we let $\bar{d},\underline{d}$ denote the lower and upper (asymptotic) densities of $A$,
\begin{align*}
	\bar{d}(A) &= \limsup_{N \to \infty} \frac{\abs{A \cap [0,N)}}{N},
	&&&
	\underline{d}(A) &= \liminf_{N \to \infty} \frac{\abs{A \cap [0,N)}}{N}.
\end{align*}
The two notions of density are connected by the relation $\underline{d}(A) = 1-\bar{d}(\NN_0 \setminus A)$. 
If $\bar{d}(A) = \underline{d}(A)$ then the common value is called the (asymptotic) density of $A$, and denoted by $d(A)$. The upper density is subadditive, meaning that 
\[ \bar{d}(A\cup B) \leq \bar{d}(A) + \bar{d}(B)\]
 for all $A,B \subset \NN_0$; the inequality can be strict even when $A$ and $B$ are disjoint. 

We will say that a statement $P(n)$ holds for \emph{(asymptotically) almost all} $n \in \NN_0$ if the set of $n \in \NN_0$ for which $P(n)$ is false has zero asymptotic density, i.e.,
\[
	\bar{d}\bra{\set{n \in \NN_0}{\neg P(n)}} = 0.
\]
More generally, if $A \subset \NN_0$ then $P(n)$ holds for \emph{almost all} $n \in A$ if the implication $n \in A \Rightarrow P(n)$ holds for almost all $n \in \NN_0$. In particular, if $\bar{d}(A) = 0$ then it is vacuously true that $P(n)$ holds for almost all $n \in A$.

For two sequences $f,g \colon \NN_0 \to \Omega$, where $\Omega$ is an arbitrary set, we will say that $f$ and $g$ are \emph{asymptotically equal}, denoted by $f \simeq g$, if $f(n) = g(n)$ for almost all $n \in \NN_0$. Note that {$\simeq$} is an equivalence relation on the set of all maps $\NN_0 \to \Omega$. Accordingly, we say that two sets $A,B \subset \NN_0$ are asymptotically equal, denoted by $A \simeq B$, if the corresponding indicator sequences are asymptotically equal, $1_A \simeq 1_B$. Equivalently, $A \simeq B$ if $\bar{d}\bra{(A \setminus B)\cup(B \setminus A)} = 0$.

We will use the following simple lemmas.
\begin{lemma}\label{lem:d_u-limit}
	Let $A_i \subset \NN_0$ for $i \in \NN_0$ be sets such that
	\begin{equation}\label{eq:30:1}
		\lim_{j \to \infty} \underline{d}\bra{\textstyle \bigcup_{i=0}^{j-1} A_i } = 1.
	\end{equation}
	Then for each set $B \subset \NN_0$ we have
	\begin{equation}\label{eq:30:2}
		\lim_{j \to \infty} \bar{d}\bra{\textstyle B \cap \bigcup_{i=0}^{j-1} A_i } = \bar{d}(B).
	\end{equation}
\end{lemma}
\begin{proof}

	It is clear that $\bar{d}\bra{ B \cap \textstyle \bigcup_{i = 0}^{j-1} A_i} \leq \bar{d}(B)$ for each $j \in \NN$, so it will suffice to show that
	\begin{equation}\label{eq:30:3}
		\bar{d}\bra{ B \cap \textstyle \bigcup_{i = 0}^{j-1} A_i } \geq \bar{d}(B).
	\end{equation}
	Since $\bar{d}$ is subadditive, for each $j \in \NN$ we have
	\begin{equation}\label{eq:30:4}
		\liminf_{j \to \infty} \bar{d}\bra{ B \cap \textstyle \bigcup_{i = 0}^{j-1} A_i } \geq \bar{d}(B) - \bar{d}\bra{\NN \setminus \textstyle \bigcup_{i = 0}^{j-1} A_i} = \bar{d}(B) - 1 + \underline{d}\bra{\textstyle \bigcup_{i = 0}^{j-1} A_i}.
	\end{equation}
	Passing to the limit $j \to \infty$ and applying \eqref{eq:30:1} yields \eqref{eq:30:3}.	
\end{proof}

\begin{lemma}\label{lem:dbar-subseq}
	Let $(N_i)_{i=1}^\infty$ be a (strictly) increasing sequence of positive integers, and assume that
	\begin{equation}
		\lambda := \limsup_{i \to \infty} \frac{N_{i+1}}{N_i} < \infty.
	\end{equation}
	Then for each set $A \subset \NN_0$ we have
	\begin{equation}\label{eq:35:1}
		 \frac{ \bar d(A) }{\lambda} \leq \limsup_{i \to \infty} \frac{\abs{A \cap [0,N_i)}}{N_i} \leq \bar d(A).
	\end{equation}
\end{lemma}
\begin{proof}
	Set $N_0 := 0$. For each $N \in \NN$, there exists $i(N) \in \NN$ such that $N \in [N_{i(N)-1}, N_{i(N)})$. Hence,
	\begin{align*}
		\bar{d}(A) & = \limsup_{N \to \infty} \frac{A \cap [0,N)}{N} 
		\leq 
		 \limsup_{N \to \infty} \frac{\abs{A \cap [0,N_{i(N)})} }{N_{i(N)}} \cdot \frac{N_{i(N)}}{N} 
		 \\& \leq
		 \lambda \limsup_{i \to \infty} \frac{\abs{A \cap [0,N_i)}}{N_i}. 
	\end{align*}
	This proves the fist of the two inequalities in \eqref{eq:35:1}; the second one is immediate.
\end{proof}

\subsection{Words} 
For $k \in \NN$ we let $\Sigma_k = \{0,1,\dots,k-1\}$ denote the set of base-$k$ digits. We let $\Sigma_k^* = \bigcup_{\alpha=0}^\infty \Sigma_k^\a$ denote the set of all words over $\Sigma_k$, including the empty word, which is denoted by $\epsilon$. For $u \in \Sigma_k^*$ we let $\abs{u}$ denote the length of $u$, and for $u,v \in \Sigma_k^*$ we let $uv$ denote the concatenation of $u$ and $v$. For $u \in \Sigma_k^*$ we let $[u]_k \in \NN_0$ denote the corresponding integer. Conversely for $n \in \NN_0$ we let $(n)_k \in \Sigma_k^*$ denote the base-$k$ expansion of $n$ (with no leading zeros). In particular, $(0)_k = \epsilon$. Slightly more generally, for $n \in \NN_0$ and $\alpha \in \NN_0$, we let $(n)_k^\a \in \Sigma_k^\a$ denote the base-$k$ expansion of $n \bmod k^\a$, padded with leading zeroes to length $\a$. Thus, for instance, $(11)_2 = (11)_2^4 = 1011$, while $(11)_2^3 = (3)_2^3 = 011$. 

For a map $\phi \colon \Sigma_k^* \to \Omega$, we define the $k$-kernel
\[
	\cN_k(\phi) = \set{ \Sigma_k^* \ni u \mapsto \phi(uv) \in \Omega}{ v \in \Sigma_k^*}.
\]
The map $\phi$ is said to be $k$-automatic if and only if its $k$-kernel is finite, $\abs{\cN_k(\phi)} < \infty$. The following lemma provides an alternative description of asymptotically $k$-automatic sequences $\NN_0 \to \Omega$, which will be helpful in applications.
\begin{lemma}\label{lem:ass-aut=>rep}
	Let $\Omega$ be a finite set, let $k \geq 2$ be an integer and let $f \colon \NN_0 \to \Omega$ be a sequence. Then the following conditions are equivalent.
\begin{enumerate}
\item\label{cond:ass-aut:A} $f$ is asymptotically $k$-automatic;
\item\label{cond:ass-aut:B} there exists $d \in \NN$, $f_0,f_1,\dots,f_{d-1} \colon \NN_0 \to \Omega$ and a $k$-automatic map $\phi \colon \Sigma_k^* \to \Sigma_d$ such that for each $u \in \Sigma_k^*$ with length $\a := \abs{u}$ we have
\begin{equation}\label{eq:f-and-phi}
		f\bra{k^\a n + [u]_k} = f_{\phi(u)}(n) \text{ for almost all } n \in \NN_0.	
\end{equation}
\end{enumerate}
\end{lemma}
\begin{proof}
	If \eqref{cond:ass-aut:B} holds then for each $f' \in \cN_k(f)$ there exists $i \in \Sigma_d$ such that $f' \simeq f_i$. Hence, $\abs{\cN_k(f)/{\simeq}} \leq d$ and in particular $f$ is asymptotically $k$-automatic. 
	
	Next, suppose that $f$ is asymptotically $k$-automatic. Let $d := \abs{\cN_k(f)/{\simeq}} < \infty$ and let $f_0,f_1,\dots,f_{d-1} \in \cN_k(f)$ be representatives of the equivalence classes. Note that for each  $u \in \Sigma_k^*$ and $\a := \abs{u}$, the sequence $f'$ given by $f'(n) := f\bra{k^\a n + [u]_k}$ belongs to $\cN_k(f)$. Hence, there exists $i \in \Sigma_d$ such that 
\begin{equation}\label{eq:f-and-phi-2}
		f\bra{k^\a n + [u]_k} = f_{i}(n) \text{ for almost all } n \in \NN_0.	
\end{equation}
Since $f_i \not \simeq f_j$ for $j \neq i$, the value of $i$ is uniquely determined by $u$. Define the map $\phi$ by setting, with the notation as above, $\phi(u) := i$. This guarantees that \eqref{eq:f-and-phi} holds, and it remains to verify that $\phi$ is $k$-automatic. 

	Let $u,v \in \Sigma_k^*$ and put $\a:= \abs{u},\ \b := \abs{v}$. Applying \eqref{eq:f-and-phi} twice, we find that
\begin{equation}\label{eq:f-and-phi-3}
		f_{\phi(uv)}(n) = f\bra{k^{\a+\b} n + [uv]_k} = f_{\phi(v)}\bra{k^\a n + [u]_k}  \text{ for almost all } n \in \NN_0.	
\end{equation}	
As a consequence, for all $v,v' \in \Sigma_k^*$ with $\phi(v') = \phi(v)$ we also have $\phi(uv) = \phi(uv')$ for all $u \in \Sigma_k^*$. In particular, $\abs{\cN_k(\phi)} \leq d$, meaning that $\phi$ is $k$-automatic.
\end{proof}
 
\section{Examples in base 2}\label{sec:Example-2}

In this section we verify the properties of sequences in Examples \ref{ex:leading-prime} and \ref{ex:gap}.

\begin{proposition}\label{prop:leading-prime}\color{checked}
	The sequence $f$ from Example \ref{ex:leading-prime} is asymptotically $2$-automatic but not asymptotically equal to a $2$-automatic sequence.
\end{proposition}
\begin{proof}\color{checked}
In order to show that $f$ is asymptotically $2$-automatic, note that for an integer $n \in \NN_0$ we have $f(2n) = f(2n+1) = f(n)$ as long as the binary expansion of $n$ contains at least one zero, i.e., as long as $n+1$ is not a power of $2$. It follows that $\abs{\cN_2(f)/{\simeq}} = 1$.

	Next, we show that $f$ is not asymptotically equal to a $2$-automatic sequence. Suppose, for the sake of contradiction, that $f \simeq \tilde f$ for a $2$-automatic sequence $\tilde f$. It follows from the pumping lemma \cite[Lemma 4.2.1]{AlloucheShallit-book} that there exists an integer $\delta \in \NN$ such that for each $\pi \in \NN$ with $\pi \geq \delta$, and for each $\alpha,m \in \NN_0$ with $m < 2^\a$ we have
	\begin{equation}\label{eq:78:1}
		\tilde f\bra{ (2^{\pi}-1)2^\a + m} = \tilde f\bra{ (2^{\pi+\delta}-1)2^\a + m}
	\end{equation}
	Pick a prime $\pi \geq \delta$ such that $\pi+\delta$ is not a prime. Then, for $\a,m \in \NN_0$ with $m < 2^{\a-1}$ we have
	\begin{align}\label{eq:78:2}
		\lini( (2^{\pi}-1)2^\a + m) &= \pi, &&& \lini( (2^{\pi+\delta}-1)2^\a + m) &= \pi+\delta,
	\end{align}
	and consequently
	\begin{align}\label{eq:78:3}
		f( (2^{\pi}-1)2^\a + m) = 1 &\neq 0 = f( (2^{\pi+\delta}-1)2^\a + m).
	\end{align}
	It follows from \eqref{eq:78:1} and \eqref{eq:78:3} that $f(n) \neq \tilde f(n)$ for exactly one $n \in \{	(2^{\pi}-1)2^\a + m, (2^{\pi+\delta}-1)2^\a + m\}$. Thus,
	\begin{align}\label{eq:78:4}
		\abs{ \set{n < 2^{\pi+\delta+\a} }{ f(n) \neq \tilde f(n)}} \geq 2^{\alpha-1}.
	\end{align}
	Letting $\a \to \infty$ (with $\delta$ and $\pi$ fixed) we conclude that
	\begin{align}\label{eq:78:5}
		\bar{d}\bra{ \set{n \in \NN_0 }{ f(n) \neq \tilde f(n)}} \geq 1/2^{\pi+\delta+1} > 0,
	\end{align}
	which contradicts the assumption that $f \simeq \tilde f$ and completes the proof.	
\end{proof}

\begin{proposition}
	The sequence $f$ from Example \ref{ex:gap} is asymptotically $2$-automatic but not asymptotically equal to a $2$-automatic sequence.
\end{proposition}
\begin{proof}
To begin with, note that $\lmax$ satisfies the following recursive relations:
	\begin{align}\label{eq:26:1}
	\lmax(2n) &= \lmax(n), &&&
	\lmax(2n+1) &= 
	\begin{cases}
		\lmax(n) + 1 & \text{if } n \equiv -1 \bmod 2^{\lmax(n)}, \\
		\lmax(n) & \text{otherwise.}
	\end{cases} 
	\end{align}

	Let $H \in \NN_0$. If $\lmax(2n+1) \neq \lmax(n)$, then it must hold that $n \equiv -1 \bmod{2^H}$ or $\lmax(n) < H$. Hence,
	\begin{align}\label{eq:26:2}
		\bar{d}\bra{\set{n \in \NN_0}{\lmax(2n+1) \neq \lmax(n)}} \leq 
		\bar{d}\bra{2^H\NN-1} + \bar{d}\bra{\set{n \in \NN_0}{\lmax(n)<H}}.
	\end{align}
	Clearly, $\bar{d}\bra{2^H\NN-1} = 2^{-H}$. It well-known that for each word $w \in \Sigma_2^*$, for almost all $n \in \NN_0$, $w$ is a subword of $(n)_2$ (see e.g.\ \cite[Thm.\ 8.6.3]{AlloucheShallit-book}). Applying this observation with $w = 1^H$, we conclude that 
	\begin{align}\label{eq:26:dens-0}	
	\bar{d}\bra{\set{n \in \NN_0}{\lmax(n)<H}} = 0.
	\end{align}
	 Inserting these values into \eqref{eq:26:2} and letting $H \to \infty$, we conclude that
	\begin{equation}
		\bar{d}\bra{\set{n \in \NN_0}{\lmax(2n+1) \neq \lmax(n)}} = 0.
	\end{equation}
	Thus, an inductive argument shows that each sequence in $\cN_2(\lmax)$ is asymptotically equal to $\lmax$. It follows that $\abs{\cN_2(f)/{\simeq}} = \abs{\cN_2(\lmax)/{\simeq}} = 1$. In particular, $f$ is asymptotically $2$-automatic.

	It remains to show that $f$ is not asymptotically equal to a $2$-automatic sequence. Suppose, for the sake of a contradiction, that $f \simeq \tilde f$ for a $2$-automatic sequence $\tilde f$. Since we have already established that
	\begin{equation}
		f(2n) = f(2n+1) = f(n) \text{ for almost all $n \in \NN_0$},
	\end{equation}
	it follows that we also have
	\begin{equation}
		\tilde f(2n) = \tilde f(2n+1) = \tilde f(n) \text{ for almost all $n \in \NN_0$}.
	\end{equation}
	As a consequence, there exists a word $w \in \Sigma_2^*$ such that 
	\begin{equation}\label{eq:49:5}
		\tilde f(2n) = \tilde f(2n+1) = \tilde f(n) \text{ for all $n \in \NN_0$ such that $w$ is a subword of $(n)_2$}.
	\end{equation}
	(This follows  e.g.\ from \cite[Thm.\ 8.6.3]{AlloucheShallit-book} applied to the $2$-automatic sequence $\NN_0 \to \Omega^3$, $n \mapsto \brabig{\tilde f(n),\tilde f(2n),\tilde f(2n+1)}$). Replacing $w$ with $1w$ if necessary, we may freely assume that $w$ begins with $1$. Let $m = [w]_2$ and $c = f(m)$.

	Our next aim is to show that $f \simeq c$. It follows from a repeated application of \eqref{eq:49:5} that 
	\begin{equation}\label{eq:49:6}
	 	\tilde f\bra{2^\a m + n}  = c \text{ for all } \alpha,n \in \NN_0 \text{ with } n < 2^\a.
	\end{equation}
	Let $\a \in \NN$ be arbitrary, and let $n < 2^{\a-1}$. If $f(n) \neq c$ then either $f(n) \neq f(2^\a m +n)$ or $f(2^\a m + n) \neq c = \tilde f(2^\a m + n)$. Since $\lmax(2^\a m + n) = \max\bra{\lmax(n),\lmax(m)}$, if  $f(n) \neq f(2^\a m + n)$ then $\lmax(n) < \lmax(m)$. Hence, 
\begin{align*}
	\abs{ \set{n \in [0,2^{\a-1}) }{f(n) \neq c} } & \leq 
	\abs{ \set{n \in [0,2^{\a-1}) }{ \lmax(n) < \lmax(m)} } 
	\\ & \phantom{\leq} +
	\abs{ \set{n' \in [0, 2^{\a}(m+1) ) }{ f(n') \neq \tilde f(n') }}.
\end{align*} 
	Dividing by $2^{\a-2}$ and letting $\a \to \infty$ we conclude that (cf.{} Lemma \ref{lem:dbar-subseq})
\begin{align*}
	\bar d\bra{ \set{n \in \NN_0 }{f(n) \neq c} } & \leq 
	2 \bar d\bra{ \set{n \in \NN_0 }{ \lmax(n) < \lmax(m)} } 
	\\ & \phantom{\leq} +
	4 (m+1) \bar d\bra{ \set{n' \in \NN_0 }{ f(n') \neq \tilde f(n') }}.
\end{align*} 	
	The first summand on the right-hand side is zero by \eqref{eq:26:dens-0}. The second summand is zero by the definition of $\tilde f$. Thus, $f \simeq c$.
	
	In order to reach a contradiction (and hence also the end of the proof) it remains to show that $f \not \simeq c$, that is, that $f$ is not asymptotically constant. This can be inferred from an explicit description of the statistical behaviour of $f$, which can be found in \cite{Tao-MO}. In order to keep the argument as self-contained as possible, we also include a short elementary proof that $f$ is not asymptotically constant.
	
	Consider the map $D \colon \NN_0 \to \NN_0$ defined as follows. Let $D(0) = 1$, and for $n > 0$, let $D(n)$ be the integer obtained by appending an additional $1$ to the earliest block of consecutive $1$s in the binary expansion of $n$ which has maximal length. To be more explicit, we may write the binary expansion of $n$ in the form
	\begin{align}\label{eq:59:1}
		(n)_2 = u 1^{\lmax(n)} v,
	\end{align}
	where $\lmax([u]_2) < \lmax(n)$, $u \in \{\epsilon\} \cup \Sigma_k^*0$, $v \in \{\epsilon\} \cup 0 \Sigma_k^*$. Note that $u$ and $v$ are uniquely determined by $n$. With this notation, we put
	\[
		D(n) := [u 1^{\lmax(n)+1} v]_2.
	\]
	
	It follows directly from the construction that $\lmax(D(n)) = \lmax(n) + 1$. In particular, $f(D(n)) \neq f(n)$. We also note that $D(n)$ contains exactly one block of $\lmax(n) + 1$ consecutive $1$s, and as a consequence $D$ is injective (indeed, $n$ can be recovered from $D(n)$ by deleting a $1$ from the longest block of consecutive $1$s). Next, for all $n > 0$ we have the estimate (using notation from \eqref{eq:59:1})
	\begin{align}\label{eq:59:2}
		D(n) = 2(n-[v]_k) + k^{\abs{v}} + [v]_k \leq 3n.
	\end{align}	 
	 Combining the observations above, we can estimate
	\begin{align*}
		N &\leq \abs{\set{n \in [0,N)}{f(n) \neq c}} + \abs{\set{n \in [0,N)}{f(D(n)) \neq c}}
		\\&\leq 2 \abs{\set{n \in [0,3N)}{f(n) \neq c}}.
	\end{align*}
	Dividing by $N$ and letting $N \to \infty$, we conclude that
	\begin{align}\label{eq:59:3}
		\underline{d}\bra{\set{n \in \NN_0}{f(n) \neq c}} \geq \frac{1}{6} & > 0. \qedhere
	\end{align}
\end{proof}

\section{Example in bases 2 and 3}\label{sec:Example}
In this section, we construct a sequence which is asymptotically $2$- and $3$-automatic, but not asymptotically equal to a periodic sequence. 

To begin with, let
\begin{equation}\label{eq:70:H_i}
(H_i)_{i=0}^\infty = \bra{2^{\a_i} 3^{\b_i} }_{i=0}^\infty 
\end{equation}
 be the increasing enumeration of the set 
\begin{equation}\label{eq:70:cH}
\cH = \set{2^\a 3^\b}{\a,\b \in \NN_0} = \{ 1,2,3,4,6,8,9,12,\dots\},
\end{equation}
 meaning that $H_i < H_{i+1}$ for all $i \in \NN_0$ and $\cH = \set{H_i}{i \in \NN_0}$. 

\begin{lemma}\label{lem:ex:1}
For the sequence $(H_i)_{i=0}^\infty$ defined above, we have
\begin{equation}\label{eq:70:1}
	\lim_{i \to \infty} \frac{H_{i+1}}{H_{i}} = 1.
\end{equation}
\end{lemma}
\begin{proof}
\renewcommand{\e}{t}
Since $\log 3/\log 2$ is irrational, it follows from Kronecker's equidistribution theorem that for each $\e > 0$ there exist $\gamma,\delta,\gamma',\delta' \in \NN_0$ such that
\[
	\gamma \log 2 - \delta \log 3,\ -\gamma' \log 2 + \delta' \log 3 \in (0,\log(1+\e)).
\]
Hence, for each $i \in \NN_0$ such that
\begin{equation}\label{eq:70:2a}
	\a_i \geq \gamma' \text{ or } \b_i \geq \delta
\end{equation}
we have
\[
	H_{i+1} \leq \max \bra{ 
	2^{\a_i + \gamma} 3^{\b_i - \delta}
	,\ 
	2^{\a_{i}-\gamma'} 3^{\beta_i + \delta'} } < (1+\e) H_i.
\]	
For fixed $\e$, \eqref{eq:70:2a} holds for all sufficiently large $i$, and hence 
\begin{equation*}
	1 \leq \liminf_{i \to \infty} \frac{H_{i+1}}{H_{i}} \leq \limsup_{i \to \infty} \frac{H_{i+1}}{H_{i}} \leq 1+\e.
\end{equation*}
Letting $\e \to 0$ we obtain \eqref{eq:70:1}.
\end{proof}

We define the sequence $f \colon \NN_0 \to \{+1,-1\}$ by $f(0) = +1$ and
\begin{equation}\label{eq:70:def-f}
	f(n) = (-1)^{\a_i+\b_i} \text{ for } n \in [H_i,H_{i+1}) \text{ and } i \in \NN_0.
\end{equation}

\begin{lemma}
	For the sequence $f$ given by \eqref{eq:70:def-f}, for almost all $ n \in \NN_0$ we have
	\begin{align}
	\label{eq:70:f-1} f(n+1) &= f(n),  \\
	\label{eq:70:f-2} f(2n) &= -f(n), \\
	\label{eq:70:f-3} f(3n) &= -f(n).
	\end{align}
\end{lemma}
\begin{proof}
\renewcommand{\e}{t}
	We begin with \eqref{eq:70:f-1}. If $f(n+1) \neq f(n)$ then, as a direct consequence of \eqref{eq:70:def-f}, we have $n +1 = H_{i}$ for some $i \in \NN_0$. For each $N \in \NN$ and $i \in \NN_0$ with $H_i \leq N$ we have $\a_i \leq (\log N)/(\log 2)$ and $\b_i \leq (\log N)/(\log 3)$, and consequently
	\begin{equation}\label{eq:70:1a}
		\abs{ \set{ n \in [0,N) }{f(n+1) \neq f(n)}  } \leq 
		\abs{ \set{ i \in \NN_0 }{H_i \leq N} } \leq  \frac{(\log N)^2}{ (\log 2)(\log 3)}.
	\end{equation}
	To finish the proof of \eqref{eq:70:f-1}, it remains to divide \eqref{eq:70:1a} by $N$ and let $N \to \infty$.
	
	Next, we prove \eqref{eq:70:f-2}. Consider an integer $n \in \NN$ with $f(2n) \neq - f(n)$. Pick $i \in \NN_0$ such that $n \in [H_i,H_{i+1})$. Since $2 \cH \subset \cH$, we have $2H_i = H_{j}$ for some $j \in \NN_0$, and likewise $2H_{i+1} = H_{j'}$ for some $j' \in \NN_0$ with $j' \geq j+1$. If $n \in [H_i,H_{j+1}/2)$ then $2n \in [H_{j},H_{j+1})$ and consequently 
	\begin{equation}\label{eq:70:2}
		f(2n) = (-1)^{\a_j + \b_j} = (-1)^{\a_{i}+1+\b_i} = -f(n),
	\end{equation}
	contrary to the choice of $n$. Thus, $n \in [H_{j+1}/2,H_{i+1})$, which in particular implies that $j' \geq j+2$ (otherwise we would have $j' = j+1$ and $2n \in [H_{j+1},H_{j+1}) = \emptyset$, which is absurd). Since $H_{i} < H_{j+1}/2 < H_{i+1}$, we see that $H_{j+1}/2$ does not belong to $\cH$, meaning that $H_{j+1}$ is an odd integer, and hence $H_{j+1}$ is a power of $3$. Note that
	\begin{equation}\label{eq:70:3}
		 H_{j+2}/H_{j+1} \leq H_{i+1}/H_i \leq 2 < 3,
	\end{equation}
	and thus $H_{j+2}$ cannot also be a power of $3$. Thus, $H_{j+2}/2 \in \cH \cap (H_i,H_{i+1}]$ and consequently $j' = j+2$.

	It follows that $H_{j+1}$ is a power of $3$, or equivalently $\a_{j+1} = 0$. Indeed, if we had $\a_{j+1} > 0$ then $H_{j+1}/2$ would be an element of $\cH$ with $H_{i} < H_{j+1}/2 < H_{i+1}$, which is absurd. A similar reasoning shows that $j' = j+2$, since $j' > j+2$ would imply that $H_{j+1}$ and $H_{j+2}$ are both powers of $3$ and consequently 
	\begin{equation}\label{eq:70:3a}
		 2 \geq H_{i+1}/H_i > H_{j+2}/H_{j+1} \geq 3.
	\end{equation}
	Thus, in particular, we have shown that
	\begin{equation}\label{eq:70:4}
		 2n \in [H_{j+1},H_{j+2}).
	\end{equation}
	\newcommand{\kk}{l}Let $K \subset \NN_0$ be the set of all integers $k \in \NN_0$ such that $H_{k}$ is a power of $3$. It follows from \eqref{eq:70:4} that for each $N \in \NN$ we have the estimate
	\begin{equation}\label{eq:70:5}
		 \abs{ \set{ n \in [0,N) }{f(2n) \neq f(n)}  } \leq 
		 \frac{1}{2} \sum_{\substack{ k \in K\\ H_k < 2N}} \abs{H_{k+1}-H_{k}}.
	\end{equation}
	Fix any $\e > 0$. By Lemma \ref{lem:ex:1}, there exists $\kk = \kk(\e)$ such that $H_{k+1}< (1+\e)H_k$ for all $k \geq \kk$. For $k < \kk$, we have the trivial estimate $H_{k+1} \leq 2 H_k$. Splitting the sum in \eqref{eq:70:5} according to whether $k \geq \kk$, for all $N > H_{\kk}/2$ we obtain
	\begin{align*}
		 \sum_{\substack{ k \in K\\ H_k < 2N }} \abs{H_{k+1}-H_{k}} 
		 & = \sum_{\substack{ k \in K\\ k < \kk }} \abs{H_{k+1}-H_{k}} +
		 \sum_{\substack{ k \in K\\ H_{\kk} \leq H_k < 2N}} \abs{H_{k+1}-H_{k}} \\
		 & \leq \sum_{\substack{ k \in K\\ k < \kk }} H_k +
		\e \sum_{\substack{k \in K \\ H_k < 2N}}  H_k \\
		 & = \sum_{\substack{ \beta < \beta_{\kk} }} 3^{\beta} +
		 \e \sum_{\substack{ \beta \leq \log_3 2N }} 3^{\beta} 
		\leq {3^{\beta_{\kk}}}/{2}  + 3\e N 
	\end{align*}	
	Inserting this estimate into \eqref{eq:70:5} and letting $N \to \infty$, we conclude that 
	\begin{equation}\label{eq:70:51}
		 \limsup_{N\to \infty} \frac{1}{N} \abs{ \set{ n \in [0,N) }{f(2n) \neq f(n)}  } \leq 
		 \frac{3}{2} \e.
	\end{equation}
	Since $\e > 0$ was arbitrary, \eqref{eq:70:f-2} follows.
		
	The proof of \eqref{eq:70:f-3} is fully analogous to the proof of \eqref{eq:70:f-2} so we do not go into its details. One minor difference is that we need to replace the estimate $2 \geq H_{i+1}/H_i$ in \eqref{eq:70:3a} with a slightly stronger estimate such as $3/2 \geq H_{i+1}/H_i$, which is easily seen to hold for all $i \geq 1$ (cf.\ Lemma \ref{lem:ex:1}).
\end{proof}

\begin{proposition}\label{prop:ex:f-is-asym-aut}
	The sequence $f$ defined by \eqref{eq:70:def-f} is asymptotically $2$-automatic and asymptotically $3$-automatic.
\end{proposition}
\begin{proof}
Pick any $f' \in \cN_2(f)$. Then there exist $\a,r \in \NN_0$ such that
\begin{equation}\label{eq:70:5:1}
	f'(n) = f(2^\a n + r) \text{ for all } n \in \NN_0.
\end{equation}
Applying \eqref{eq:70:f-1} iteratively $r$ times, we conclude that
\begin{equation}\label{eq:70:5:2}
	f'(n) = f(2^\a n) \text{ for almost all } n \in \NN_0.
\end{equation}
Next, applying \eqref{eq:70:f-2} iteratively $\a$ times, we conclude that
\begin{equation}\label{eq:70:5:3}
	f'(n) = (-1)^\a f(n) \text{ for almost all } n \in \NN_0.
\end{equation}
Thus, $f' \simeq f$ or $f' \simeq -f$, depending on the parity of $\a$. It follows that $\cN_2(f)/{\simeq}$ has (at most) two elements. In particular $\cN_2(f)/{\simeq}$ is finite and $f$ is  asymptotically $2$-automatic. The proof that $f$ is asymptotically $3$-automatic is fully analogous.
\end{proof}

\begin{proposition}\label{prop:ex:f-not-periodic}
	The sequence $f$ defined by \eqref{eq:70:def-f} is not asymptotically equal to a periodic sequence
\end{proposition}
\begin{proof}
Suppose, for the sake of contradiction, that there exists a periodic sequence $\tilde f \colon \NN_0 \to \{+1,-1\}$  
 \begin{equation}\label{eq:70:6:1}
	f(n) = \tilde f(n) \text{ for almost all } n \in \NN_0.
\end{equation}
It follows from \eqref{eq:70:6:1} and \eqref{eq:70:f-1} that 
\begin{equation}\label{eq:70:6:2}
	\tilde f(n) = \tilde f(n+1) \text{ for almost all } n \in \NN_0.
\end{equation}
Since $\tilde f$ is periodic, it follows that \eqref{eq:70:6:2} holds for all $n \in \NN_0$, and hence $\tilde f$ is constant. Pick $c \in \{+1,-1\}$ such that $\tilde f(n) = c$ for all $n \in \NN_0$. Then
 \begin{equation}\label{eq:70:6:1B}
	f(n) = c \text{ for almost all } n \in \NN_0.
\end{equation}
Combining \eqref{eq:70:6:1B} and \eqref{eq:70:f-2}, we conclude that
 \begin{equation}\label{eq:70:6:3}
	f(n) = -f(2n) = -c \text{ for almost all } n \in \NN_0.
\end{equation}
Since $c \neq -c$, \eqref{eq:70:6:3} and \eqref{eq:70:6:1B} yield a contradiction, which finishes the argument.
\end{proof}

\section{Non-examples in all bases}\label{sec:NonExple}

We now discuss the example showing that asymptotic shift-invariance does not imply asymptotic automaticity.

\begin{proposition}\label{prop:non-ex}
	Let $k \geq 2$ and let $f \colon \NN_0 \to \{0,1\}$ be given by
	\[
		f(n) = \ip{\sqrt{n}} \bmod 2.
	\]
	Then $f$ is not asymptotically $k$-automatic.
\end{proposition}
\begin{proof}
	It will suffice to show that for all $\a,\b \in \NN_0$ with $\a \neq \b$, for sequences
	\begin{align*}
		f'(n) &:= f(k^\a n), &&& f''(n) &:= f(k^\b n), &&& (n \in \NN_0) 
	\end{align*}
	we have $f' \not \simeq f''$. Indeed, once this is shown, it will immediately follow that $\cN_k(f)/{\simeq}$ is infinite. For ease of notation, we restrict to the case where $\a = 1$, $\b = 0$ and $k \geq 5$ is not a square; the general case requires no new ideas.
	
	Suppose, for the sake of contradiction, that $f' \simeq f'' = f$. It follows that $f(n) = f(\ip{n/k})$ for almost all $n \in \NN$. For $i \in \NN_0$, the squares in the interval $[ki^2,k(i+1)^2)$ are
	\[  \ceil{\sqrt{k}i}^2, \bra{  \ceil{\sqrt{k}i}  +1} ^2, \dots,  \floor{\sqrt{k}(i+1)}^2; \]
	note that under our assumptions, this includes at least three squares. For all $n \in [ki^2,k(i+1)^2)$ we have $f\bra{\ip{n/k}} = i \bmod 2$. On the other hand, for each of $c \in \{0,1\}$ we have 
	\[	
		f(n) = c 
		\begin{cases}
		\text{ for all } n \in \left[ \ceil{\sqrt{k}i}^2,\bra{ \ceil{\sqrt{k}i}+1 }^2\right) 
		&\text{ if } \ceil{\sqrt{k}i} \equiv c \bmod{2},\\
		\text{ for all } 
		n \in \left[\bra{\ceil{\sqrt{k}i}+1}^2,\bra{\ceil{\sqrt{k}i}+2}^2 \right)
		&\text{ if } \ceil{\sqrt{k}i} \equiv 1-c \bmod{2}.
		\end{cases}\]
	In particular, we can bound the number of places where $f(n) \neq f(\ip{n/k})$ by
	\[
		\abs{ \set{n \in [ki^2,k(i+1)^2)}{f(n) = c} } \geq 2\ceil{\sqrt{k}i}+1. 
	\]	
	As a consequence, we have
	\[
		\frac{ \abs{\set{n \in [ki^2,k(i+1)^2)}{f(n) \neq f(\ip{n/k)}}} }{\abs{[ki^2,k(i+1)^2)}}
		\geq \frac{2 \ceil{\sqrt{k}i}+1}{k(2i+1)} = \frac{1}{\sqrt{k}} \cdot \bra{1+ O({1}/{i})}.
	\]
	Separating an arbitrary interval $[0,N)$ into intervals of the form $[ki^2,k(i+1)^2)$ and a negligible remainder set, we conclude that
	\[
		\bar{d}\bra{\set{n \in \NN_0}{f(n) \neq f(\ip{n/k})}} \geq \frac{1}{\sqrt{k}} > 0.
	\]
	This contradicts the assumption that $f' \simeq f$ and finishes the argument.
\end{proof}

Since for the sequence $f$ from Proposition \ref{prop:non-ex} we have $f(n) = f(n+1)$ unless $n+1$ is a square, we have finished verifying the claims made in Example \ref{ex:non}. We also point out that there is nothing special about the map $x \mapsto \sqrt{x}$ which appears in the definition of $f$: A very similar argument would work for any map of the form $f(n) = \ip{\varphi(n)} \bmod 2$ , where $\varphi \colon \RR \to \RR$ is sufficiently regular increasing function with sub-linear growth.

\section{Proofs of the main results}\label{Sec:Proofs}

\begin{proof}[Proof of Theorem \ref{thm:main-symmetric}]\color{checked}
Pick representatives 
\[ f_0,f_1,\dots,f_{d-1} \in \cN_k(f) \text{ and } g_0,g_1,\dots,g_{e-1} \in \cN_l(f)\]
of the equivalence classes in $\cN_k(f)/{\simeq}$ and $\cN_l(f)/{\simeq}$ respectively. Put 
\[ \vec f = (f_0,f_1,\dots,f_{d-1}) \colon \NN_0 \to \Omega^d \text{ and } \vec g = (g_0,g_1,\dots,g_{e-1}) \colon \NN_0 \to \Omega^e.\]
Let $\phi \colon \Sigma_k^* \to \Sigma_d$ and $\psi \colon \Sigma_l^* \to \Sigma_e$ be the maps from Lemma \ref{lem:ass-aut=>rep} applied to $f$ in bases $k$ and $l$ respectively. Let $\alpha_0$ be an integer such that $\phi(u0^{2\a_0}v) = \phi(u 0^{\a_0} v)$ for all $u,v \in \Sigma_k^*$, whose existence can be deduced, for instance, repeating the argument in \cite[ Lemma 2.2]{ByszewskiKonieczny-2020-AA}. Note that, more generally, we have
\begin{align}\label{eq:798:a:1}
	\phi(u0^{\a}v) = \phi(u 0^{\a_0} v) \text{ for all } u,v \in \Sigma_k^* \text{ and } \a \in \a_0 \NN.
\end{align}

Let $\phi' \colon \NN_0 \to \Sigma_d$ be the map defined for $m \in \NN_0$ by $\phi'(m) = \phi\bra{0^\mu (m)_k}$, where $\mu$ is the unique integer (dependent on $m$) such that $\a_0 \leq \mu < 2\a_0$ and $\mu + \abs{(m)_k)} \equiv 0 \bmod \a_0$. This definition guarantees that for all $\a \in \a_0 \NN$  for all and $m \in \NN_0$ with $\a \geq 2\a_0$, and $m < k^{\a-2\a_0}$ we have
\begin{equation}\label{eq:798:f-phi}
	f(k^\a n + m) = f_{\phi( (m)_k^{\a} )}(n) 
= f_{\phi'(m)}(n) \text{ for almost all } n \in \NN_0.
\end{equation}	
Crucially, the expression on the right-hand side of \eqref{eq:798:f-phi} does not depend on $\a$.
Let $\beta_0 \in \NN$ and $\psi' \colon \NN_0 \to \Sigma_e$ be defined analogously to $\a_0$ and $\phi'$, using base $l$ rather than $k$.

For $\vec x \in \Omega^d$ and $\vec y \in \Omega^e$ consider the sets
\begin{align}
\label{eq:798:3-AB}	
A(\vec x) &:= \set{ n \in \NN_0 }{ \vec f(n) = \vec x}, &&&
B(\vec y) &:= \set{ n \in \NN_0 }{  \vec g(n) = \vec y}.
\end{align}
Let $X_0$ denote the set of all $\vec x \in \Omega^d$ with $\bar d(A(\vec x)) > 0$, and let $Y_0 \subset \Omega^e$ be defined analogously.
Note that we have
\begin{align}
	\label{eq:798:4} \bigcup_{\vec x \in X_0} A(\vec x) \simeq \bigcup_{\vec x \in \Omega^d} A(\vec x) &= \NN_0, &&&
\bigcup_{\vec y \in Y_0} B(\vec y) \simeq \bigcup_{\vec y \in \Omega^e} B(\vec y) &=  \NN_0.
\end{align}
More generally, for $\a,\b \in \NN_0$ let $X_{\b} \subset X_0$  and $Y_\a \subset Y_0$ denote the sets
\begin{align}\label{eq:79:def-Xb}
	X_\b &:= \set{\vec x \in \Omega^d}{\bar d\bra{A\bra{\vec x} \cap l^\beta \NN_0} > 0}, 
	&&& 
	Y_\a &:= \set{\vec y \in \Omega^e}{\bar d\bra{B\bra{\vec y} \cap k^\a \NN_0} > 0}.
\end{align}
Next, we put
\begin{align}\label{eq:79:def-X*}
	X_* &:= \bigcap_{\b \in \b_0 \NN_0} X_\b, &&& Y_* &:= \bigcap_{\a \in \a_0 \NN_0} Y_\a.
\end{align}

\begin{lemma}\label{lem:internal}
For each $\vec x \in X_*$, the sequence $m \mapsto x_{\phi'(m)}$ is eventually periodic.
\end{lemma}
\begin{proof}
Let $\alpha \in \alpha_0 \NN$ and $\beta \in \beta_0 \NN$. Directly from \eqref{eq:79:def-X*}, we have
\begin{align}\label{eq:79:5:1}
	\bar{d}\bra{A(\vec x) \cap l^\b \NN_0} > 0.
\end{align}
In other words, recalling that $A(\vec x)/l^\b := \set{n \in \NN_0}{l^\b n \in A(\vec x)}$, we have
\begin{align}\label{eq:79:5:2}
	\bar{d}\bra{A(\vec x)/l^\b} > 0.
\end{align}
By the same argument as in \eqref{eq:798:4}, we obtain 
\begin{align}\label{eq:79:5:3}
	\NN_0 \simeq \bigcup_{\vec y \in Y_\a} B(\vec y)/k^\a.
\end{align}
Hence, there exists $\vec y \in Y_\a$ such that 
\begin{align}\label{eq:79:5:5}
	\bar{d}\braBig{\bra{A(\vec x)/l^\b} \cap \bra{B(\vec y)/k^\a}} > 0.
\end{align}

Consider any integer $m$ with 
\begin{equation}\label{eq:798:m}
0 \leq m < \min\bra{ k^{\a-2\a_0}, l^{\b-2\b_0} }.
\end{equation} 
Bearing in mind \eqref{eq:798:f-phi} and \eqref{eq:798:3-AB}, we have
\begin{equation}\label{eq:798:f-1}
	f( k^\a l^\b n + m) = f_{\phi((m)_k^\a)}(l^\b n) = x_{\phi'(m)} \text{ for almost all } n \in A(\vec x)/l^\b.
\end{equation}
Similarly, applying the same reasoning in base $l$, we have
\begin{equation}\label{eq:798:g-1}
	f(k^\a l^\b n + m) = g_{\psi((m)_l^\b)}(k^\a n) = y_{\psi'(m)}  \text{ for almost all } n \in B(\vec y)/k^\a.
\end{equation}
Combining \eqref{eq:798:f-1} and \eqref{eq:798:g-1}, we conclude that
\begin{equation}\label{eq:798:x=y}
	x_{\phi'(m)} = f(k^\a l^\b n + m) = y_{\psi'(m)}  \text{ for almost all } n \in \bra{A(\vec x)/l^\b} \cap \bra{B(\vec y)/k^\a}.
\end{equation}
It follows from \eqref{eq:79:5:5} there exists at least one $n$ for which \eqref{eq:798:x=y} holds, and thus
\begin{equation}\label{eq:798:x=y!}
x_{\phi'(m)} = y_{\psi'(m)}.
\end{equation}
Since $\a \in \a_0\NN$ and $\b \in \b_0 \NN$ are arbitrary and \eqref{eq:798:m} is the only restriction on $m$, we conclude that \eqref{eq:798:x=y!} holds for all $m \in \NN_0$. 

The expression on the left-hand side of \eqref{eq:798:x=y!} is a $k$-automatic sequence with respect to $m$, while the expression on the right-hand side is an $l$-automatic sequence. It now follows from the classical version of Cobham's theorem that the sequence in \eqref{eq:798:x=y!} is eventually periodic. 
\end{proof}

Let $q$ be the least common multiple of the periods of the sequences $m \mapsto x_{\phi'(m)}$, where $\vec x \in X_*$. (Since $X_* \subset \Omega^d$ is finite, the definition of $q$ is well-posed.) Likewise, let $m_0$ be the least point beyond which the aforementioned sequences are periodic. Thus, for each $\vec x \in X_*$ and $m \geq m_0$ we have 
\(
	x_{\phi'(m+q)} = x_{\phi'(m)}.
\)
If we additionally assume that $m < k^{\a-2\a_0}-q$ then, by \eqref{eq:798:f-phi},
\begin{equation}\label{eq:798:f-is-q-inv}
	f(k^\a n + m+q) = x_{\phi'(m+q)} = x_{\phi'(m)} = f(k^\a n+m) \text{ for almost all } n \in A_\a(\vec x).
\end{equation}

Let us consider the set
\[
	E := \set{n \in \NN_0 }{ f(n+q) \neq f(n)}.
\]
We aim is to show that $\bar{d}(E) = 0$. By \eqref{eq:798:f-is-q-inv}, we already know that
\begin{equation}\label{eq:79:7:1}
	\bar{d}\bra{ E \cap \bra{k^\a A(\vec x)  + [m_0,k^{\a-2\a_0}-q) }} = 0
\end{equation}
for each $\a \in \a_0 \NN$ and $\vec x \in X_*$. By Lemma \ref{lem:d_u-limit}, it will suffice to show that
\begin{equation}\label{eq:79:7:2}
	\underline{d}\bra{\bigcup_{\a_0 \mid \a < \gamma} \bigcup_{\vec x \in X_*} \bra{k^\a A(\vec x)  + [m_0,k^{\a-2\a_0}-q) }} \to 1 \text{ as } \gamma \to \infty.
\end{equation}
(Above, the union is taken over $\a \in \a_0 \NN$ with $\a < \gamma$.) To bring \eqref{eq:79:7:2} into a slightly more convenient form, let $\delta \in \NN_0$ be the least integer such that $m_0 \leq k^\delta$ and $k^{-2\a_0} \geq k^{-\delta} + q k^{-2\delta}$; this ensures that for each $\a \in \a_0\NN$ we have
\begin{equation}\label{eq:79:71:1}
	[k^\delta, k^{\a - \delta}) \subset 	[m_0,k^{\a-2\a_0}-q).
\end{equation}
(Note that the interval on the left hand side is empty when $\a \leq 2\delta$.)
Thus, it will suffice to show that 
\begin{equation}\label{eq:79:7:3}
	\underline{d}\bra{\bigcup_{\a_0 \mid \a < \gamma} \bigcup_{\vec x \in X_*} \bra{k^\a A(\vec x)  + [k^{\delta},k^{\a-\delta}) }} \to 1 \text{ as } \gamma \to \infty.
\end{equation}

\begin{lemma}\label{lem:internal-2}
	There exists $\b_* \in \beta_0 \NN$ such that 
	\begin{equation}\label{eq:79:beta*}
	\bar{d}\bra{ l^{\b_*} \NN_0 \setminus \bigcup_{\vec x \in X_*} A(\vec x) } = 0.
	\end{equation}
\end{lemma}
\begin{proof}
	Note that for each $\b \in \b_0\NN$ we have $X_{\b+\b_0} \subset X_\b$. Thus, $(X_{i\b_0})_{i=0}^\infty$ is a descending sequence of finite sets, and hence an eventually constant sequence. Hence, there exists $\b_* \in \beta_0\NN$ such that $X_* = X_{\b_*}$. In other words, for each $\vec x \in \Omega^d$, if $\bar{d}\bra{l^{\b_*} \NN \cap A(\vec x)} > 0$ then $\vec x \in X_*$. Bearing in mind \eqref{eq:798:4}, it follows that 
	\begin{align}
	\bar{d}\bra{ l^{\b_*} \NN_0 \setminus \bigcup_{\vec x \in X_*} A(\vec x) } \leq \sum_{\vec x \in \Omega^d \setminus X_*} \bar{d}\bra{ l^{\b_*} \NN_0 \cap A(\vec x) } = 0,
	\end{align}
	as needed.
\end{proof}

Applying Lemma \ref{lem:internal-2} to the union in \eqref{eq:79:7:3} we obtain
\begin{equation}\label{eq:79:8:1}
	\underline{d}\bra{ \bigcup_{\a_0 \mid \a < \gamma} \bigcup_{\vec x \in X_*} \bra{k^\a A(\vec x)  + [k^{\delta},k^{\a-\delta}) }} 
	\geq \underline{d}\bra{\bigcup_{\a_0 \mid \a < \gamma} \bra{k^\a l^{\b_*} \NN_0  + [k^\delta,k^{\a-\delta}) }}.
\end{equation}
Thus, in order to finish the argument, it will suffice to show that
\begin{equation}\label{eq:79:8:2}
	\underline{d}\bra{\bigcup_{\a_0 \mid \a < \gamma} \bra{k^\a l^{\b_*} \NN_0  + [k^\delta,k^{\a-\delta}) }} \to 1 \text{ as } \gamma \to \infty.
\end{equation}
It will be convenient to prove a slightly more general fact.

\begin{lemma}\label{lem:dens-union-1}
	Let $k,m,\delta \in \NN$, $k\geq 2$. Then
	\begin{align}
		\underline{d}\bra{ \bigcup_{\a=0}^{\gamma-1} \bra{m k^\a \NN_0 + [k^{\delta},k^{\a-\delta})} } \to 1 \text{ as } \gamma \to 1.
	\end{align}
\end{lemma}
\begin{proof}
	Note that, for any $\lambda \in \NN$, we may freely replace $k$ with $k^\lambda$ and $\delta$ with $\ceil{\delta/\lambda}$. Picking sufficiently large $\lambda$, we may thus assume that $m < k/2$ and $\delta = 1$.

	Bearing in mind Lemma \ref{lem:dbar-subseq}, it will suffice to show that 
	\begin{align}
		\lim_{\gamma \to \infty} \liminf_{\nu \to \infty} k^{-\nu} \abs{ [0,k^\nu) \cap \bigcup_{\a=0}^{\gamma-1} \bra{m k^\a \NN + [k^\delta,k^{\a-\delta}) } }= 1.
	\end{align}
	
	Pick any $\gamma,\nu \in \NN$ with $\nu \geq \gamma \geq 3$. Consider the random variable $\bb n$, uniformly distributed on $[0,k^\nu)$, and let 
	\[
		\bb u = \bb u_{\nu-1} \bb u_{\nu-2} \dots \bb u_1 \bb u_0 = (\bb n)_k^\nu.
	\]
	Then $\bb u_0,\bb u_1, \dots, \bb u_{\nu-1}$ are jointly independent random variables, each uniformly distributed in $\Sigma_k$. For $\a \in \NN$ with $3 \leq \a < \gamma$, let $\cA_{\a}$ be the event that 
\begin{enumerate}[wide]
\item[(\emph{i})] $[\bb u_{\nu-1} \bb u_{\nu-2} \dots \bb u_{\a}]_k \equiv 0 \bmod m$;
\item[(\emph{ii})] $\bb u_{\a-1} = 0$;
\item[(\emph{iii})] $\bb u_{\a-2}\bb u_{\a-3}\dots\bb u_{1} \neq 0$.
\end{enumerate}
We have
\begin{align}\label{eq:87:1}
k^{-\nu}  \abs{ [0,k^\nu) \cap \bigcup_{\a=0}^{\gamma-1} \bra{m k^\a \NN + [k^\delta,k^{\a-\delta}) } } = 
\PPP\bra{\cA_\a \text{ for at least one } \a \in [3, \gamma) }.
\end{align}
Let $\cA_\a'$ be the event that $\cA_\a$ holds and additionally 
\begin{enumerate}[wide]
\item[(\emph{iii}')] $\bb u_{\a-2} \neq 0$.
\end{enumerate}
Since condition (\emph{iii}') clearly implies (\emph{iii}), the event $\cA_\a'$ depends only on $\bb u_{\nu-1} \bb u_{\nu-2} \dots \bb u_{\a-2}$.

For each $\alpha$ with $3 \leq \a < \gamma$, conditional on $\bb u_{\nu-1} \bb u_{\nu-2} \dots \bb u_{\a+1}$, we can ensure that $\cA_{\a}'$ holds by assigning suitable values to $\bb u_{\a},\bb u_{\a-1}, \bb u_{\a-2}$, and thus
\begin{align*}
\PPP\bra{\cA_\a' \middle| \bb u_{\nu-1} \dots \bb u_{\a+1} }
&\geq \PPP\bra{
\begin{array}{l}
[\bb{u}_{\a}]_k \equiv - [\bb u_{\nu-1} \dots \bb u_{\a+1}0]_k \bmod{m},\\ 
\bb{u}_{\a} \neq 0,\ \bb{u}_{\a-1} = 0,\ \bb{u}_{\a-2} \neq 0
\end{array}
\middle| \bb u_{\nu-1} \dots \bb u_{\a+1}} 
\\& \geq
 \frac{\ip{k/m}-1 }{k} \cdot \frac{1}{k} \cdot \frac{k-1}{k} =: p > 0.
\end{align*}
Applying this inequality with $\a = \gamma-1,\gamma-4,\dots$, we conclude that
\begin{align}\label{eq:87:2}
\PPP\bra{\cA_\a \text{ for at least one } \a \in [3,\gamma) } \geq 1 - (1-p)^{\ip{\gamma/3} - 1} \to 1 \text{ as } \gamma \to \infty.
\end{align}
Combining \eqref{eq:87:1} and  \eqref{eq:87:2} completes the argument.
\end{proof}

Applying Lemma \ref{lem:dens-union-1} with $k' = k^\a$, $m = l^{\b_*}$ and $\delta' = \ceil{\delta/\a}$ we conclude that \eqref{eq:79:8:2} holds, which together with previous considerations completes the proof.
\end{proof}

\begin{proof}[Proof of Theorem \ref{thm:main-asymmetric}]
	It follows from Theorem \ref{thm:main-symmetric} that $f$ is asymptotically invariant under shifts by $q$ for some $q \in \NN$. Replacing $f(n)$ with $f(qn+r)$ for arbitrary $r \in \{0,1,\dots,q-1\}$, we may freely assume that $q=1$ (cf.\ \cite[Thm.\ 6 .8.2]{AlloucheShallit-book} ). We claim that $f$ is asymptotically constant. Pick any $c \in \Omega$ such that the preimage $f^{-1}(c)$ has positive upper density. Then (e.g.{} by \cite[Lemma 2.3]{ByszewskiKonieczny-2020-AA}) there exists some $m \in \NN$ and $\a_0 \in \NN$
	\begin{align}
	\label{eq:304:a} f(n) &= c & \text{ for all } n \in [k^\a m, k^\a (m+1)) \text{ and } \a \in \a_0 \NN.
	\end{align}
	In other words, letting $K = k^{\a_0}$, $\mu_0 := \log_K(m)$ and $\lambda := \log_K\bra{(m+1)/{m}}$, we have
	\begin{align}
	\label{eq:305:a} f(n) &= c & \text{ for all } n \in \NN \text{ with } \log_K n \in [\mu_0,\mu_0+\lambda) + \ZZ.
	\end{align}

Because $f$ is asymptotically $l$-automatic, using the pigeon-hole principle, we can find $\beta \in \NN_0$ and $\delta \in \NN$ such that
\begin{align}\label{eq:306:1}
	f(l^{\beta + \delta}n + 1)& = f(l^\beta n + 1) & \text{ for almost all } n \in \NN_0.
\end{align}
It follows from \eqref{eq:305:a} and \eqref{eq:306:1} that
	\begin{align}
	\label{eq:306:2} f(n) &= c & \text{ for almost all } n \in l^{\beta}\NN_0+1 \text{ with }\log_K(l^\delta n) \in [\mu_0,\mu_0+\lambda)+ \ZZ.
	\end{align}
Since $f(n+1) = f(n)$ for almost all $n$, we can strengthen \eqref{eq:306:2} to
	\begin{align}
	\label{eq:306:3} f(n) &= c & \text{ for almost all } n \in \NN_0 \text{ with }\log_K(n) \in [\mu_1,\mu_1+\lambda)+ \ZZ, 
	\end{align}	
where $\mu_1 = \mu_0 - \delta \log_K(l)$. Iterating this reasoning, for each $h \in \NN$ we obtain
	\begin{align}
	\label{eq:306:3-h} f(n) &= c & \text{ for almost all } n \in \NN_0 \text{ with }\log_K(n) \in [\mu_h,\mu_h+\lambda)+ \ZZ, 
	\end{align}	
where $\mu_h := \mu_0 - h \delta \log_K(l)$. Since $ \log_K(l) \in \RR \setminus \QQ$, we can find a finite sequence $h_1,h_2,\dots,h_s \in \NN_0$ such that 
\[
	\textstyle \bigcup_{i=1}^s \left[\mu_{h_i} , \mu_{h_i} + \lambda \right) + \ZZ = \textstyle \bigcup_{i=1}^s \left[\mu_0 - h_i \delta \log_K(l) , \mu_0 - h_i \delta \log_K(l)  + \lambda \right) + \ZZ = \RR.
\]
As a consequence, $f(n) = c$ for almost all $n \in \NN$, as needed.
\end{proof} 

\bibliographystyle{alphaabbr}
\bibliography{bibliography}

\end{document}